\documentclass[reqno, 12pt]{amsart}   
\usepackage{graphicx}
\usepackage{amsmath,amssymb,amsthm, mathrsfs, mathtools, booktabs, framed}
\usepackage[foot]{amsaddr}

\usepackage{xcolor}
\usepackage[backref]{hyperref}
\hypersetup{
    colorlinks,
    linkcolor={red!60!black},
    citecolor={green!60!black},
    urlcolor={blue!60!black}
}
\usepackage[T1]{fontenc}
\usepackage{lmodern}
\usepackage[babel]{microtype}
\usepackage[english]{babel}
\usepackage{comment}

\usepackage{geometry}
\usepackage{pgfplotstable} 

\usepackage{soul}

\numberwithin{equation}{section}
\numberwithin{figure}{section}

\usepackage{enumerate}

\theoremstyle{plain}
\newtheorem{thm}{Theorem}[section]
\newtheorem{theorem}[thm]{Theorem}
\newtheorem{prop}[thm]{Proposition}

\newtheorem{cor}[thm]{Corollary}

\newtheorem{lemma}[thm]{Lemma}

\renewcommand{\leq}{\leqslant}
\renewcommand{\geq}{\geqslant}
\renewcommand\le{\leqslant}
\renewcommand\ge{\geqslant}

\newcommand{\quadric}{\ensuremath{\mathcal{E}}}
\newcommand\largestnumber{97\relax}

\newcommand{\F}{\mathbb{F}}
\newcommand{\PG}{\mathrm{PG}}
\newcommand{\tr}{\mathrm{tr}}
\newcommand{\cone}{\mathrm{cone}}
\newcommand{\even}{\mathbb{E}}

\usepackage{listings}

\definecolor{dkgreen}{rgb}{0,0.6,0}
\definecolor{gray}{rgb}{0.5,0.5,0.5}
\definecolor{mauve}{rgb}{0.58,0,0.82}

\title{On Bruen chains}

\author[Bamberg]{John Bamberg$^1$}
\author[Lansdown]{Jesse Lansdown$^{1,2}$}
\address{$^1$Department of Mathematics and Statistics, 
The University of Western Australia, Crawley, WA 6009, Australia.}

\author[Van de Voorde]{Geertrui Van de Voorde$^2$}
\address{$^2$School of Mathematics and Statistics, University of Canterbury, 
Christchurch, New Zealand}

\email{john.bamberg@uwa.edu.au}
\email{\{jesse.lansdown,geertrui.vandevoorde\}@canterbury.ac.nz}

\date{}

\begin{document}

\begin{abstract}
It is known that a Bruen chain of the three-dimensional projective space $\PG(3,q)$ exists for every odd prime power $q$ at most $37$, except for $q=29$. It was shown by Cardinali et. al (2005) that Bruen chains do not exist for $41\le q\leq 49$. We develop a model, based on finite fields, which allows us to extend this result to $41\le q \le \largestnumber$, thereby adding more evidence to the conjecture that Bruen chains do not exist for $q>37$. 
Furthermore, we show that Bruen chains can be realised precisely as the $(q+1)/2$-cliques of a two related, yet distinct, undirected simple graphs. 
\end{abstract}

\maketitle

\section{Introduction}

There is a functorial correspondence between translation planes of order $q^2$, with a kernel of order $q$,
and \emph{spreads} of the three-dimensional projective space $\PG(3,q)$, in the sense that isomorphic translation planes correspond to projectively equivalent spreads of $\PG(3,q)$. A spread of $\PG(3,q)$ is a set of lines forming a partition of the points, and the finite Desarguesian projective planes $\PG(2,q)$ correspond to the \emph{regular} spreads of $\PG(3,q)$. Taking a regulus inside a regular spread and replacing it with its opposite regulus yields a \emph{Hall} plane, and geometers have investigated line replacement as a method for constructing new projective planes.
Bruen \cite{Bruen1978} introduced the concept of a \emph{chain} of reguli in a regular spread of $\PG(3,q)$, $q$ odd; a
set of $(q+3)/2$ reguli that pairwise meet in 2 lines, such that every line
of $\PG(3,q)$ lies in 0 or 2 reguli of the chain. A Bruen chain is essentially the 
smallest set of reguli one needs in a partial spread replacement construction, to create a non-Desarguesian translation plane having rank 2 over its kernel, such that the set of reguli contains at least two intersecting reguli.

For every odd prime power $q$ between 5 and 37,
but apart from 29, there exists a Bruen chain, and there are 17 in total (up to semi-similarity; see Table \ref{tbl:known}). 
It was shown by Cardinali et al. \cite{CDPT} that there are no examples for $41\le q\le 49$.
We extend their result by showing the following:

\begin{theorem}\label{computationalresult}
There are no Bruen chains for $41\le q\le \largestnumber$.
\end{theorem}

The above result relies on computation and we provide code for setting up the computation in an Appendix.
One of our key observations is that a Bruen chain is equivalent to a $(q+1)/2$-clique in a certain undirected simple graph
$\Gamma_X$, which we describe as follows. 
Given an elliptic quadric $\quadric$
of $\PG(3,q)$,
there are two classes of external points of $\quadric$ up to isometry, so without loss of generality choose $\even$ to be one such class. Let $X\in \even$ 
and let $\mathbb{V}$ be the set of elements of $\even$ not equal to $X$ and lying on 
an external line with $X$. The set $\mathbb{V}$ is the vertices of $\Gamma_X$, and
we join two vertices $A$ and $B$ if the line $AB$ does not contain $X$ and the cones $\cone(X)$, $\cone(A)$, $\cone(B)$ intersect only in points of $\quadric$. Here, $\cone(A)$ is the union of $A$ with the set of points $Y\neq A$ such that the line $AY$ is tangent to $\quadric$.
Furthermore, we define the related graph $\Delta_X$ by removing the condition that the line $AB$ does not contain $X$. Thus $\Gamma_X$ is a spanning subgraph of $\Delta_X$.

In this paper, we show the following:
\begin{theorem}\label{maintheorem}
Cliques of size $(q+1)/2$ of $\Gamma_X$ and cliques of size $(q+1)/2$ of $\Delta_X$ are in one-to-one correspondence with Bruen chains.
\end{theorem}
We will give more information about cliques of smaller size in Theorem \ref{maintheoremextra}.

Theorem \ref{maintheorem} shows that the question of the existence of Bruen chains can be solved by finding
 the clique-number of $\Gamma_X$ or $\Delta_X$. 
In this paper we develop an efficient model for both graphs based on finite fields. This model 
allows us to compute data which shows that the clique number attains the theoretical upper bound only for the small graphs in the family, but the clique number is relatively low once the order of the graph exceeds a certain threshold.

\section{Bruen Chains as sets of points in $\PG(3,q)$}
\subsection{Using the polarity}

Let $q$ be an odd prime power, $q\ge 5$, then an alternative model for Bruen chains arises when we consider an elliptic quadric $\quadric$ of $\PG(3,q)$. Recall that every plane meets $\quadric$ in either one point, or the set of $q+1$ points of a (non-degenerate) conic; the latter set is called a \emph{conic of $\quadric$}.
A \emph{Bruen chain} is then a set $\mathcal{C}$ of $(q+3)/2$ conics of $\quadric$  such that
\begin{itemize}
\item[(i)] any two elements of $\mathcal{C}$ intersect in two points, and
\item[(ii)] any three elements of $\mathcal{C}$ have empty intersection. 
\end{itemize}

By using the polarity arising from $\quadric$, denoted by $\perp$, we can instead consider the `polar' version of a Bruen chain, which we will use instead.

\begin{lemma}\label{bruenPerp}
Let $\mathcal{C}$ be a Bruen chain of the elliptic quadric $\quadric$.
Apply the polarity arising from $\quadric$ to the planes spanning each element of $\mathcal{C}$,
to obtain $(q+3)/2$ external points $\mathcal{C}^\perp$. Then $\mathcal{B}=\mathcal{C}^\perp$ satisfies the following properties,
and conversely, any set of external points satisfying these properties is the polar image of a Bruen chain.
\begin{enumerate}[(i)]
\item Any two elements of $\mathcal{B}$ span an external line;
\item Any three elements of $\mathcal{B}$ span a secant plane.
\end{enumerate}
In particular, the set of points in $\mathcal{B}$ forms a {\em cap} (i.e. a set of points no three of which are collinear).
\end{lemma}

\begin{proof}
Suppose two conics $C$ and $C'$ of $\quadric$ meet in two points $A$ and $B$. 
Then the polar images of the secant planes $\langle C\rangle$
and $\langle C'\rangle$, are external points $X$ and $X'$. The line $AB$
is a secant line, and so its polar image, $XX'$ is an external line. So the equivalent condition to 
the first property of a Bruen chain is that the two external points $X$ and $X'$ of $\mathcal{B}$ span
an external line.

If a tangent plane $X^\perp$ to $\quadric$ contains the external points $\langle C\rangle^\perp$, $\langle C'\rangle^\perp$, $\langle C''\rangle^\perp$, then the point $X\in \quadric$ is contained in $\langle C\rangle\cap \langle C'\rangle\cap \langle C''\rangle$ and vice versa. Hence, condition (ii) that three conics of a Bruen chain $\mathcal{C}$ intersect trivially translates into the condition that no tangent plane to $\quadric$ contains three points of $\mathcal{B}$.
Finally, note that the condition that no tangent plane contains three points of $\mathcal{B}$ is equivalent to condition (ii) in the Lemma: if three elements of $\mathcal{B}$ are collinear, then they are contained in two tangent planes to $\quadric$. Hence, if the three points are not contained in a tangent plane, they span a plane which is a secant plane.
\end{proof}

Lemma \ref{bruenPerp} shows that no tangent plane contains more than two points of a Bruen chain (seen as external points to $\quadric$). The following lemma sharpens that result.

\begin{lemma}\label{tangentplanes}
Every tangent plane to $\quadric$ intersects a Bruen chain (as external points to $\quadric$) in 0 or 2 elements.
\end{lemma}

\begin{proof}
Let $\mathcal{B}$ be a Bruen chain.
First, by Lemma \ref{bruenPerp}(ii), every tangent plane meets $\mathcal{B}$ in
at most 2 points. Let $X\in \mathcal{B}$, and consider the set of tangent planes incident
with $X$. There are $q+1$ such tangent planes, and for every $Y\in\mathcal{B}\backslash\{X\}$,
there are precisely two tangent planes incident with $XY$ (because $XY$ is an external line).
Since $|\mathcal{B}\backslash\{X\}|=(q+1)/2$, it follows by direct double-counting, that every tangent plane on $X$ intersects $\mathcal{B}\backslash\{X\}$ in precisely 1 element.
The result then follows.
\end{proof}

There are two isometry classes of external points (according to the quadratic form's value being square or nonsquare respectively), and Heden \cite{Heden95} showed that a Bruen chain is contained entirely within one isometry class of external points. We will give a short proof of this fact at the end of this section (see Corollary \ref{Heden_main}).
Moreover, under the full group of semi-similarities, the two isometry classes may be interchanged and so there is a group element mapping a Bruen chain on external points from one isometry class to another. As a result, a Bruen chain exists in one isometry class on external points if and only if a Bruen chain exists in both classes.

\subsection{A finite field model}\label{sect:model}

We will present a model of the elliptic quadric that gives an alternative
algebraic characterisation of Bruen chains, in a similar vein to Section 4 of \cite{Heden95}.
Consider the finite field $\F_{q^4}$ as a vector space over its subfield $\F_q$.
Let $\tr$ be the relative trace map from $\F_{q^4}$ to $\F_q$, that is, the $\F_q$-linear map $\tr(x)=x+x^q+x^{q^2}+x^{q^3}$. The mapping $Q(x)=\tr(x^2)$ is a quadratic form on $\F_{q^4}$ of elliptic type\footnote{One way to see 
that the form is of elliptic type
is to notice that $Q$ is the `field-reduction' of the quadratic form $x\mapsto x^2$, from $\F_{q^4}$ to $\F_{q^4}$. The latter form
has no nonzero singular vectors, and so is elliptic.}.
The elliptic quadric is determined by the elements of $\F_{q^4}^*$ such that
$Q(x)=0$, where we consider elements up to scalars in $\F_q$. We will denote the  point in $\PG(3,q)$ defined by the element $x\in\F_{q^4}^*$ by $\langle x\rangle$. By abuse of notation, the line (resp. plane) spanned by $\langle \delta\rangle,\langle\epsilon\rangle$ (resp. $\langle \delta\rangle,\langle\epsilon\rangle, \langle \gamma\rangle$) is denoted by $\langle \delta,\epsilon\rangle$ (resp. $\langle \delta,\epsilon,\gamma\rangle$.
An isometry $g$ for $Q$ is an invertible $\F_q$-linear map on $\F_{q^4}$
such that $Q(x^g)=Q(x)$ for all $x\in\F_{q^4}$. The group of isometries 
for $Q$ has three orbits on points $\langle x\rangle$ in $\PG(3,q)$: the singular points, `even' external points,
and `odd' external points. Accordingly, these orbits are characterised by the value of $Q(x)$
being zero, nonzero square, or nonsquare in $\F_q$. Note that this notion is well-defined: since $\langle x\rangle=\langle y\rangle$ if and only if $y=\lambda x$ for some $\lambda\in \F_q$, we have that $Q(x)$ is a (non-)square if and only if $Q(y)$ is a (non-)square. 
We say that points in the same orbit have the same \emph{isometry type}.

\begin{lemma}\label{extviafield}
Two external points $\langle\alpha\rangle$ and $\langle\beta\rangle$ in $\PG(3,q)$ span an external line if and only if $\tr(\alpha\beta)^2-\tr(\alpha^2)\tr(\beta^2)$ is a nonsquare. They span a tangent line if and only if $\tr(\alpha\beta)^2-\tr(\alpha^2)\tr(\beta^2)=0$.
\end{lemma}

\begin{proof}
A point lying on the line spanning $\langle\alpha\rangle$ and $\langle\beta\rangle$, and different from $\langle\beta\rangle$, is of the form
$\langle\alpha+\lambda\beta\rangle$ for some $\lambda\in \F_q$. This point lies in the quadric
if and only if $\tr\left( (\alpha+\lambda \beta)^2 \right)=0$. Equivalently, 
$\tr(\alpha^2)+2\lambda \tr(\alpha\beta)+\lambda^2 \tr(\beta^2)=0$ which is a quadratic equation in $\lambda$.
Since $q$ is odd, this equation has zero solutions if the discriminant 
$\tr(\alpha\beta)^2-\tr(\alpha^2)\tr(\beta^2)$ is a nonsquare, and exactly one solution if the discriminant is zero. The result follows.
\end{proof}

\begin{cor}\label{tangenttype} 
Two external points $\langle\alpha\rangle$ and $\langle\beta\rangle$, contained in a tangent line $\ell$ to $\quadric$, have the same isometry type.
\end{cor}

\begin{proof} The points $\langle\alpha\rangle$ and $\langle\beta\rangle$ span a tangent line if and only if $\tr(\alpha\beta)^2=\tr(\alpha^2)\tr(\beta^2)$. Hence, $\tr(\alpha^2)$ and $\tr(\beta^2)$ are both squares or both nonsquares.
\end{proof}

\begin{lemma}\label{tangentplane}
Let $\gamma,\delta,\epsilon$ be elements of $\F_{q^4}$ such that $\gamma,\delta,\epsilon$ are $\F_q$-independent 
and such that $\tr(\gamma^2)\tr(\delta^2)\tr(\epsilon^2)\neq 0.$
Then $\langle\gamma\rangle,\langle\delta\rangle,\langle\epsilon\rangle$ define a tangent plane to $\quadric$
if and only if
 \[\tr(\gamma^2)\tr(\delta^2)\tr(\epsilon^2)-\tr(\gamma^2)\tr(\delta \epsilon)^2-\tr(\delta^2)\tr(\epsilon\gamma)^2-\tr(\epsilon^2)\tr(\gamma \delta)^2+2\tr(\gamma \delta)\tr(\delta \epsilon)\tr(\epsilon\gamma)=0.\]
\end{lemma}

\begin{proof} Suppose that $\langle\gamma\rangle,\langle\delta\rangle,\langle\epsilon\rangle$ determine a tangent plane $\pi$ to $\quadric$. Without loss of generality (possibly relabeling $\delta,\epsilon,\gamma$), we may assume that the tangent point $\pi\cap\quadric$ is not contained in the line $\langle\delta,\epsilon\rangle$, so it is of the form
$\langle\gamma+\lambda_1\delta+\lambda_2\epsilon\rangle$ for some $\lambda_1,\lambda_2\in \F_q$. This point lies in $\quadric$
if and only if $\tr\left( (\gamma+\lambda_1\delta+\lambda_2\epsilon)^2 \right)=0$. Equivalently, 
\[
\tr(\gamma^2)+2 \tr(\delta \gamma) \lambda_1 + \tr(\delta^2) \lambda_1^2 + 
 2 \tr(\epsilon \gamma) \lambda_2 + 2 \tr(\delta \epsilon) \lambda_1 \lambda_2 + 
 \tr(\epsilon^2)\lambda_2^2=0
\]
which is a quadratic equation in $\lambda_1$ since $\tr(\delta^2)\neq 0$.
The discriminant is
\[
D:=4 \left(-\tr(\delta^2) \left(\lambda_2^2 \tr(\epsilon^2)+2 \lambda_2 \tr(\epsilon \gamma)+\tr(\gamma^2)\right)+2 \lambda_2 \tr(\delta \epsilon) \tr(\delta \gamma)+\lambda_2^2 \tr(\delta \epsilon)^2+\tr(\delta \gamma)^2\right).
\]
In order for there to be a unique solution in $\lambda_1$, we require $D$ to be 0. This yields
a quadratic equation in $\lambda_2$ since $\tr({\epsilon^2})\ne 0$, whose discriminant is
\[
64 \tr(\delta^2)\left(-\tr(\delta^2) \tr(\epsilon^2) \tr(\gamma^2)+\tr(\epsilon^2) \tr(\delta \gamma)^2+
\tr(\gamma^2) \tr(\delta \epsilon)^2+\tr(\delta^2) \tr(\epsilon \gamma)^2-2 \tr(\delta \epsilon) \tr(\delta \gamma) \tr(\epsilon \gamma)\right).
\]
This discriminant is also required to be zero, which is, since $\tr(\delta^2)\neq 0$, equivalent to the statement given in the theorem. Vice versa, we see that if the condition of the theorem holds, there is exactly one point of $\quadric\cap \pi$, not contained on the line spanned by $\delta$ and $\epsilon$. Since $\pi$  either meets $\quadric$ in a conic or in one point, and $q>2$, it follows that $|\pi\cap \quadric|=1$.
\end{proof}

Denoting the polarity associated to $\quadric$ by $\perp$, we see that $\cone(A)$ is the cone with vertex $A$ and base the conic $A^\perp\cap\quadric$. We also call $\cone(A)$ the cone \emph{defined by $A$}.
Take $\mathbb{E}$ to be the set of even points and recall that the isometry group acts transitively on the points of $\mathbb{E}$. 
Since $\tr(1^2)=4$ is a square, $1\in \mathbb{E}$, and  in what follows, we will take $X=\langle 1\rangle$.

\begin{lemma} \label{conegeneral} Let $X=\langle 1\rangle$ and let $\delta,\epsilon$ be elements of $\F_{q^4}$ such that $1,\delta,\epsilon$ are $\F_q$-independent and $\tr(\delta)\tr(\epsilon)\neq 0$. Let
\[f(x,y)=\tr(y)x-\tr(x)y.\]
Let $\Omega$ be the set of  points, not contained in $\quadric$, that are contained in the intersection of the cone defined by $X$ and the planes $\langle \delta\rangle^\perp:\tr(\delta x)=0$ and $\langle\epsilon\rangle^\perp:\tr(\epsilon x)=0$.
Then $|\Omega|\geq 1$ if $-\tr(f(\delta,\epsilon)^2)$ is zero or a nonsquare.
More precisely:
\begin{itemize}
\item[(i)]
If $\langle 1, \delta, \epsilon\rangle$ is not a tangent plane to $\quadric$, then $|\Omega|=2$ if 
$-\tr(f(\delta,\epsilon)^2)$ is a nonsquare; $|\Omega|=1$ if $\tr(f(\delta,\epsilon)^2)=0$, and $|\Omega|=0$ if $-\tr(f(\delta,\epsilon)^2)$ is a non-zero square.

\item[(ii)] If $\langle 1, \delta, \epsilon\rangle$ is a tangent plane to $\quadric$, then $|\Omega|=1$. 
\end{itemize}

\end{lemma}

\begin{proof} Since $\delta$ and $\epsilon$ are $\F_q$-independent, the points $x$ satisfying $\tr(\delta x)=0=\tr(\epsilon x)$ determine a line $\ell$ in $\PG(3,q)$.
The intersection points of the cone defined by $X$ and the two planes determined by $\tr(\delta x)=0$ and $\tr(\epsilon x)=0$ are precisely the points of $\ell$ contained in the cone of $X$. It follows from Lemma \ref{extviafield} that the cone defined by $X$ has equation 
$$\tr(x)^2-4\tr(x^2)=0.$$
Since $\tr(\delta)\tr(\epsilon)\neq 0$, $\ell$ does not contain $X$.
Furthermore, $\ell$ does not lie in $X^\perp$: suppose $\ell\subseteq X^\perp$, then $X$ lies in $\ell^\perp$, the line spanned by $\langle \delta\rangle$ and $\langle \epsilon\rangle$. But this would imply that $1,\delta,\epsilon$ are $\F_q$-dependent.
Let $m$ be the projection of the line $\ell$ from $X$ onto the plane $X^\perp$ and let $\mathcal{C}$ be the conic $X^\perp\cap \quadric$; we have just shown that $m$ is a line different from $\ell$.

First note that the plane $\langle 1,\delta,\epsilon\rangle$ is a tangent plane to $\quadric$ if and only if the point $\ell\cap X^\perp\in \mathcal{C}$: suppose that $Y\in \ell\cap X^\perp\cap \mathcal{C}$.
Since $\ell$ contains $Y$, $Y^\perp$ contains $\ell^\perp$, and since $Y\in X^\perp$, $X\in Y^\perp$. Since $1,\delta,\epsilon$ are $\F_q$-independent, we find that $Y^\perp$, the tangent plane in the point $Y$ of $\quadric$, is spanned by $\langle1\rangle,\langle\delta\rangle,\langle\epsilon\rangle$. Vice versa, if $\langle 1,\delta,\epsilon\rangle$ is a tangent plane meeting $\quadric$ in a point $P$, then $P^\perp=\langle X,\ell^\perp\rangle$, implies that $P\in X^\perp$ and $P\in \ell$, so indeed $P=\ell\cap X^\perp\in \mathcal{C}$.

We distinguish between the cases that $\langle 1,\delta,\epsilon\rangle$ is a secant plane or a tangent plane. First assume that $\langle 1,\delta,\epsilon\rangle$ is a secant plane, so, equivalently, assume that $\ell$ does not meet $X^\perp$ in a point of $\mathcal{C}$. 

We now distinguish between the cases in which $m$, (which is the line $\langle X,\ell\rangle\cap X^\perp$) is external, tangent or secant to $\mathcal{C}$. It is clear that the number of intersection points of the line $\ell$ with the cone determined by $X$ is then $0,1$ or $2$ respectively.
 We will first show that if the line $m$ is tangent to $\mathcal{C}$ then $\tr(f(\delta,\epsilon)^2)=0$.
 So suppose that $m$ is a tangent line to $\mathcal{C}$ in the point $P_0=\langle x_0\rangle$. Since $P_0$ is contained in $X^\perp$, $X\in P_0^\perp$ and $P_0^\perp$ is the plane spanned by $X$ and $m$. So $\ell\in P_0^\perp$, which implies that $P_0$ in $\ell^\perp$. Since $\ell^\perp$ is spanned by $\langle \delta\rangle$ and $\langle \epsilon\rangle$, it follows that $x_0=\delta+\nu \epsilon$ for some $\nu\in \F_q$. We have that $\tr(x_0)=\tr(\delta)+\nu\tr(\epsilon)=0$ so $\nu=-\frac{\tr(\delta)}{\tr(\epsilon)}$.
We conclude that $x_0=\delta-\frac{\tr(\delta)}{\tr(\epsilon)}\epsilon$.
Now $f(\delta,\epsilon)=\tr(\epsilon)\delta-\tr(\delta)\epsilon=\tr(\epsilon)x_0$.

Since $P_0\in \quadric$, we have that $\tr(x_0^2)=0$, and hence, since $\tr(\epsilon)\neq 0$, $-\tr(f(\delta,\epsilon)^2)=-\tr(\epsilon)^2\tr(x_0^2)=0$. 
 The line $m$, contained in $X^\perp$, is secant to $\mathcal{C}$ when $m$ is a secant to $\quadric$. This happens if and only if $m^\perp$ is an external line to $\quadric$. The line $m$ is the intersection of the plane $X^\perp$ with the plane spanned by $X$ and $\ell$ which implies that the line $m^\perp$ is precisely the span of $X$ with the point, say $K$, which is the intersection of $X^\perp$ with $\ell^\perp$.
 Since $\ell$ is given by $\tr(\delta x)=0$ and $\tr(\epsilon x)=0$, the line $\ell^\perp$ is spanned by $\delta$ and $\epsilon$. This implies that $K$ is a point $\langle\kappa\rangle$ with $\kappa=\lambda\delta +\mu \epsilon$ such that $\tr(\kappa)=0$. It follows that we can take $\kappa=\tr(\epsilon)\delta-\tr(\delta)\epsilon$. Note that $\kappa\neq 0$ since $\delta,\epsilon$ are $\F_q$-independent and $X$ does not lie on $\ell$, so $\tr(\delta)$ and $\tr(\epsilon)$ cannot both be zero.

 It follows that the points on $m^\perp$, the span of $X$ with $\kappa$, are given by $\mu+\kappa$ for some $\mu\in \F_q$, and that $m^\perp$ is an external line to $\quadric$ if and only if $\tr((\mu+\kappa)^2)=0$ has no solution for $\mu$. Using that $\tr(\kappa)=0$, this equation is equivalent to
 $4 \mu^2+\tr(\kappa^2)=0$,
 which has no solutions if and only if $-\tr(\kappa^2)$ is a nonsquare. We conclude that if $-\tr(\kappa^2)$ is a nonsquare, the line $\ell$ and the cone determined by $X$ have two intersection points. The same reasoning, reversing the roles of secant/external shows that if $-\tr(\kappa^2)$ is a non-zero square, there are no admissible intersection points between the cone determined by $X$ and the line $\ell$.

 Now assume that $\langle 1,\delta,\epsilon\rangle$ is a tangent plane, so equivalently, that $\ell$ intersects $X^\perp$ in a point of $\mathcal{C}$, say $Y$.
If $\langle 1,\delta,\epsilon\rangle$ is a tangent plane, then $\tr(\delta^2)\neq 0$: if $\tr(\delta^2)=0$, then $\langle \delta\rangle\in \quadric$ and $\langle 1,\delta,\epsilon\rangle$ would be the tangent plane $\langle \delta\rangle^\perp$, a contradiction since $\tr(\delta)\neq 0$ implies that $1\notin \langle \delta\rangle^\perp$. Hence, we can apply Lemma \ref{tangentplane}, to find that the condition that $\langle 1,\delta,\epsilon\rangle$ 
 is a tangent plane is equivalent to
\[4\tr(\delta^2)\tr(\epsilon^2)-4\tr(\delta \epsilon)^2-\tr(\delta^2)\tr(\epsilon)^2-\tr(\epsilon^2)\tr( \delta)^2+2\tr(\delta)\tr(\delta \epsilon)\tr(\epsilon)=0.\]

Now we can simplify this to
$4(\tr(\delta^2)\tr(\epsilon^2)-\tr(\delta \epsilon)^2)-\tr(f(\delta,\epsilon)^2)=0$
and so
\[
-\tr(f(\delta,\epsilon)^2)=4(\tr(\delta \epsilon)^2-\tr(\delta^2)\tr(\epsilon^2)).
\]
Since the plane $\langle 1,\delta,\epsilon\rangle$ is a tangent plane,  the line $\ell^\perp$ is  external to $\quadric$, and hence, by Lemma \ref{extviafield}, $\tr(\delta \epsilon)^2-\tr(\delta^2)\tr(\epsilon^2)$ is a nonsquare. It follows that $-\tr(f(\delta,\epsilon)^2)$ is a nonsquare. In particular, it shows that $-\tr(f(\delta,\epsilon)^2)\neq 0$.
We now show that this also implies that the line $m$ is secant to $\mathcal{C}$. Recall that $m$ contains at least one point of $\mathcal{C}$, namely $\ell\cap X^\perp$. So suppose to the contrary that $m$ is the tangent line to $\mathcal{C}$ in $\mathcal{\ell}\cap X^\perp$, then
 $Y^\perp$ contains $\ell$, in which case $\ell\cap\ell^\perp=Y$, and hence, $\ell^\perp\cap X^\perp$ in $\pi$. As in part (i), expressing that $Y$ lies on $\ell^\perp$ shows that $Y=\langle f(\delta,\epsilon)\rangle$, and $Y\in \quadric$ if and only if $-\tr(f(\delta,\epsilon)^2)=0$, a contradiction.
We conclude that $m$ is a secant to $\mathcal{C}$, and hence, since $\ell\cap X^\perp=m\cap X^\perp=Y$, that $\ell$ meets the cone defined by $X$ in precisely one point, not contained in $X^\perp$.
\end{proof}

\begin{prop}[The Cone Condition]\label{conecondition}
Let $X=\langle 1\rangle$, $A=\langle\alpha\rangle$, $B=\langle \beta\rangle$ with $\tr(\alpha^2)=a^2$ and $\tr(\beta^2)=b^2$, and let
\[f(x,y)=\tr(y)x-\tr(x)y.\]
Suppose that $XA$ and $XB$ are external lines to $\quadric$.
Then the cones defined by $X,A,B$ have a common (external) point if 
$X,A,B$ are noncollinear and
$-\tr(f(x,y)^2)$ is zero or a nonsquare for some
$(x,y)\in\{(2\alpha+a,2\beta+b),(2\alpha-a,2\beta+b),(2\alpha+a,2\beta-b),(2\alpha-a,2\beta-b)\}$.
\end{prop}
\begin{proof} It follows from Lemma \ref{extviafield} that the cone determined by $X$ consists of the points $\langle x\rangle$ such that $\tr(x)^2-4\tr(x^2)=0$, and similarly, the cone determined by $A=\langle \alpha\rangle$ is determined by those $\langle x\rangle$ with $\tr(\alpha x)^2-\tr(\alpha^2)\tr(x^2)=0$. 
Now consider the system determined by both equations:
\begin{align}
    \tr(x)^2-4\tr(x^2)&=0 \label{treq1}\\
    \tr(\alpha x)^2-\tr(\alpha^2)\tr(x^2)&=0\label{treq2}
\end{align}
Substituting $\tr(x^2)=\frac{\tr(x)^2}{4}$ from \eqref{treq1} into \eqref{treq2}, and using $\tr(\alpha^2)=a^2$, we find that this system is equivalent with
\begin{align*}
    \tr(x)^2-4\tr(x^2)&=0 \\
    4\tr(\alpha x)^2-a^2\tr(x)^2&=0
\end{align*} and hence, with
\begin{align*}
    \tr(x)^2-4\tr(x^2)&=0 \\
    \tr((2\alpha+a)x)\tr((2\alpha-a)x)&=0.
\end{align*}
Repeating the same argument for the cone determined by $B$, we find that the system defined by the cones determined by $X,A$ and $B$ is given by
\begin{align}
    \tr(x)^2-4\tr(x^2)&=0 \label{cone1}\\
    \tr((2\alpha+a)x)\tr(2\alpha-a)x)&=0\label{cone2}\\
    \tr((2\beta+b)x)\tr(2\beta-b)x)&=0\label{cone3},
\end{align}
which splits into four different systems defined by the cone determined by $X$ intersected with a line of the form $\tr(\delta x)=0=\tr(\epsilon x)=0$, where $(\delta,\epsilon)\in \{(2\alpha+a,2\beta+b),(2\alpha-a,2\beta+b),(2\alpha+a,2\beta-b),(2\alpha-a,2\beta-b)\}$. 
Note that $1,2\alpha\pm a,2\beta\pm b$ are $\F_q$-independent since $X,A,B$ are not collinear.
Furthermore, if $\tr(2\alpha\pm a)=0$, then $4\tr(\alpha)^2=a^2=\tr(\alpha^2)$. But by Lemma \ref{extviafield}, this implies that $XA$ is a tangent to $\quadric$, a contradiction. Similarly, $\tr(2\beta\pm b)\neq 0$. Hence, we can apply Lemma \ref{conegeneral} and the statement follows.
\end{proof}

\subsection{Recasting Heden's results}

For self-containment, we provide new proofs that use our finite field model of some of Heden's results on Bruen chains
\cite{Heden95}. We shortcut some of the arguments of Heden and give a short proof of his
main result that the elements of a Bruen chain are of the same isometry type.

\begin{prop}[{\cite[Proposition 3.3]{Heden95}}]\label{sameclass}
Let $X$ and $X'$ be two external points spanning an external line, and 
let $\pi$ be a tangent plane not through $X$ nor $X'$. Then
\begin{enumerate}[(i)]
    \item $X$ and $X'$ lie in the same isometry class if and only if $|\cone(X)\cap \cone(X')\cap \pi|=2$.
    \item If $X$ and $X'$ lie in different isometry classes, then $|\cone(X)\cap \cone(X')|=0$.
\end{enumerate}
\end{prop}

\begin{proof}
(i) Without loss of generality, we may suppose
$X=\langle 1\rangle$ and $X'=\langle \alpha\rangle$, $\tr(\alpha^2)=a^2$, and such that 
$\tr(\alpha)^2-4\tr(\alpha^2)$ is a nonsquare. 
We may write the tangent plane $\pi$ as zeroes of the equation $\tr(\gamma x)=0$
where $x$ is variable, and $\gamma$ is a fixed element of $\F_{q^4}$ such that $\tr(\gamma^2)=0$. Since $X\notin \pi$, $\tr(\gamma)\neq 0$, and since $X'\notin \pi$, $\tr(\alpha \gamma)\neq 0$. 

As in Proposition \ref{conecondition}, we can rewrite the intersection of the cone defined by $X$ and the cone defined by $X'$ by the intersection of the cone defined by $X$ with the set $\tr((2\alpha+a)x)\tr((2\alpha-a)x)=0$, with  $\tr(2\alpha\pm a)\neq 0$ since $XX'$ is an external line. If $1,2\alpha\pm a,\gamma$ would be $\F_q$-independent, then $X,X',\langle \gamma\rangle$ would be collinear, a contradiction because $\langle \gamma\rangle$ is a point of $\quadric$. 
Furthermore, the points $X,X',\langle \gamma\rangle$ do not span a tangent plane, since the unique tangent plane through $\langle \gamma\rangle$ is $\pi$, and $X\notin \pi$.
Hence, the conditions of Lemma \ref{conegeneral}(i) for $\delta=2\alpha\pm a$ and $\epsilon=\gamma$ are satisfied.
The intersection points of $|\cone(X)\cap \cone(X')\cap \pi|$ are precisely those of $|\cone(X)\cap \langle(2\alpha+a)\rangle^\perp \cap \langle\gamma\rangle^\perp|$ and $|\cone(X)\cap \langle(2\alpha-a)\rangle^\perp \cap \langle\gamma\rangle^\perp|$.
Now $f(2\alpha+a,\gamma)$ equals $\tr(\gamma)(2\alpha+a)-\tr(2\alpha+a)\gamma$, and it follows that $-\tr(f(2\alpha+a,\gamma)^2)=8\tr(\alpha\gamma)\tr(\gamma)(\tr(\alpha)+2a)$ and 
$\tr(f(2\alpha-a,\gamma)^2)=8\tr(\alpha\gamma)\tr(\gamma)(\tr(\alpha)-2a)$.
The product $\tr(f(2\alpha+a,\gamma)^2)\tr(f(2\alpha-a,\gamma)^2)$ equals $$64\tr(\alpha\gamma)^2\tr(\gamma)^2(\tr(\alpha)^2-4a^2).$$  Since $XX'$ is an external line, $\tr(\alpha)^2-4a^2$ is a nonsquare. Furthermore, $\tr(\gamma)\neq 0$ and $\tr(\alpha\gamma)\neq 0$ so the product is non-zero. It follows that exactly one of $-\tr(f(\alpha+2a,\gamma)^2)$ and $-\tr(f(\alpha-2a,\gamma)^2)$ is a nonsquare, while the other is a square, so Lemma \ref{conegeneral} shows that there are exactly $2$ points in $|\cone(X)\cap \cone(X')\cap \pi|$.

(ii) If $Y$ is a point in $\cone(X)\cap \cone(X')$, then the lines $XY$ and $X'Y$ are both tangent lines to $\quadric$. By Corollary \ref{tangenttype} this implies that $X,Y$ and $Y'$ have the same isometry type, a contradiction.
\end{proof}

\begin{lemma}[{\cite[Lemma 3.1]{Heden95}}]\label{evenincidences}
Let $P$ be an external point not in a Bruen chain $\mathcal{B}$.
Then the cone on $P$ intersects $\mathcal{B}$ in an even number of elements.
\end{lemma}

\begin{proof}
Let $\mathcal{S}$ be the set of pairs of the form $(\pi,A)$ where $\pi$ is a tangent plane on $P$
and $A$ is a point of $\mathcal{B}$ incident with $\pi$.
There are $q+1$ tangent planes $\pi$ through $P$, and each one intersects $\mathcal{B}$
in 0 or 2 elements by Lemma \ref{tangentplanes}. Hence, $|\mathcal{S}|$ is even. A point $A$ lies in $0$, $1$, or $2$ tangent planes on $P$,
so let $C_i$ be the number of points of $\mathcal{B}$ lying
on $i$ tangent planes on $P$. Therefore, $|\mathcal{S}|=0\cdot C_0+1\cdot C_1+2\cdot C_2$
which is even. Therefore, $C_1$ is even.
\end{proof}

\begin{prop}[{\cite[Proposition 3.4(b)]{Heden95}}]\label{coneintersectBruen}
Let $P$ be an external point not in a Bruen chain $\mathcal{B}$. Then the cone on $P$ intersects $\mathcal{B}$ in precisely two elements. Conversely, given two elements $A,B\in\mathcal{B}$, there exist precisely two cones containing them.
\end{prop}

\begin{proof}
Let $R$ be a point of the Bruen chain $\mathcal{B}$, and let $\pi$ be a tangent plane not incident with $R$.
Let $T$ be the set of external points of $\pi$ lying in $\cone(R)$. So $|T|=q+1$.
We will count pairs $(R',t)$ where $R'\in\mathcal{B}\backslash\{R\}$, and $t\in T$, such that
$t$ lies in $\cone(R')$.

First, given $R'$, there are 0 or 2 elements $t$ such that $t\in \cone(R)\cap \cone(R')\cap \pi$, by Proposition \ref{sameclass}(i). Since the number of possibilities of $R'$ is $(q+1)/2$, 
the number of pairs $(R',t)$ is at most $2(q+1)/2=q+1$.

In the other direction, given $t\in T$, there is at least one $R'\in \mathcal{B}\backslash\{R\}$
such that $R'$ lies in the cone determined by $T$, by Lemma \ref{evenincidences}. So
the number of pairs $(t,R')$ is at least $q+1$. So overall, we have shown that the lower and upper bounds
on the number of pairs $(R',t)$ are equal, and so there are two consequences: (i) Given $R'$, 
there are precisely 2 elements $t$ such that $t\in \cone(R)\cap \cone(R')\cap \pi$; (ii) Given $t\in T$,
there are precisely two elements of the Bruen chain (including $R$) lying in the cone of $t$.
\end{proof}

\begin{cor}[{\cite[Proposition 3.1]{Heden95}}]\label{Heden_main}
Every element of a Bruen chain is of the same isometry type.
\end{cor}

\begin{proof} Proposition \ref{coneintersectBruen} shows that two elements, say $A$ and $B$ of a Bruen chain are contained in two cones. If the cone defined by $C$ contains $A$ and $B$, then the lines $AC$ and $BC$ are tangent to $\quadric$. By Corollary \ref{tangenttype}, $A$ and $C$, and $B$ and $C$ have the same isometry type; hence $A$ and $B$ have the same type.
\end{proof}

The following result is new and connects Bruen chains to the graphs $\Gamma_X$.

\begin{cor}\label{Bruenimpliescone}
Let $X$, $A$, $B$ be three elements of a Bruen chain. Then the intersection
$\cone(X)\cap \cone(A)\cap \cone(B)$ has no points outside $\quadric$.
\end{cor}

\begin{proof} If there were such a point $R$ in the intersection of the three cones, the cone determined by $R$ would contain the three points $X,A,B$, a contradiction by Proposition \ref{coneintersectBruen}.
\end{proof}

\section{Bruen chains as subsets of graphs}

As in the introduction, let $\mathbb{V}$ be the set of external points in the isometry
class $\even$, not equal to $X$, and lying on some external line with $X$. We have two graphs on $\mathbb{V}$ described below, and we will interpret intersection of cones as taking the common external points of the cones (that is, common points, not contained in $\quadric$).
\begin{center}
\begin{tabular}{ll}
\toprule
Graph & Adjacency relation\\
\midrule
$\Delta_X$ & $A\sim B\iff \cone(X)\cap \cone(A)\cap \cone(B)=\varnothing$\\
$\Gamma_X$ & $A\sim B\iff \cone(X)\cap \cone(A)\cap \cone(B)=\varnothing$ and $X\notin AB$\\
\bottomrule 
\end{tabular}
\end{center}
Obviously, $\Gamma_X$ is a spanning subgraph of $\Delta_X$. We also point out that
the condition $A\sim B\iff \cone(X)\cap \cone(A)\cap \cone(B)=\varnothing$ implies that
$AB$ span an external line (see Lemma \ref{coneimpliesexternal}). These two graphs are very different to those given 
by \cite[\S3]{CDPT}, as their adjacency relation 
is just `$A$ and $B$ span an external line and $X\notin AB$' in this context.

For $q=1\pmod{4}$, there are no points in $\mathbb{V}$ that are orthogonal to $X$, because if we take a finite field element $a$ with $\tr(a^2)\in\square$ orthogonal to $1$, then 
$\tr(a\cdot 1)^2-\tr(a^2)\tr(1^2)=-4\tr(a^2)\in \square$. So $\langle a\rangle$ would not span an external line with $X=\langle 1\rangle$. This is the main reason why the graph $\Gamma_X$ that we have defined is a regular graph for $q=1\pmod{4}$, but has two distinct vertex degrees for $q=3\pmod{4}$. 

Now the points of $\even$ lying on an external line on $X$ give rise
to a $(q-1)/2$-clique of $\Delta_X$, whereas, the same set is a coclique of $\Gamma_X$.
We show that $\Delta_X$ always has a coclique of size ${q\choose 2}$.

\begin{lemma}\label{coclique2}
Take an external point $Y$ on some tangent line $\ell$ on $X$. There are $q$ tangent lines to $\quadric$ on $Y$, different from $\ell$. Each tangent line has $(q-1)/2$ of the vertices of $\mathbb{V}$ incident with it. The union of the vertices of $\mathbb{V}$ on the $q$ tangent lines is a coclique of $\Delta_X$ of size $q(q-1)/2$.
\end{lemma}

\begin{proof}
 Note that, since $X$ is an even point (i.e. $X\in \even$) and $Y$ lies on a tangent line to $\quadric$ through $X$, the point $Y$ is even by Lemma \ref{tangenttype}. This implies that all points on the cone defined by $Y$ (but not contained in $\quadric$) are even as well. Furthermore, the line $XY$ does not contain any vertices of $\Delta_X$. Let $m$ be one of the $q$ tangent lines through $Y$, different from $XY$. Then $\langle \ell,m\rangle$ meets $\quadric$ in a conic $\mathcal{C}$. Since $X$ lies on precisely $(q-1)/2$ external lines to $\mathcal{C}$, it follows that $m$ contains precisely $(q-1)/2$ points $P$ for which $PX$ is an external line to $\mathcal{C}$. Hence, $|\mathcal{S}|=\frac{q(q-1)}{2}$. It is clear from the construction that for two points $Z_1,Z_2$ in $\mathcal{S}$, the cones defined by $X,Z_1,Z_2$ have (at least) the point $Y$ in common. Hence, $Z_1$ and $Z_2$ are not adjacent in $\Delta_X$.
\end{proof}

\begin{prop}
The clique number of $\Delta_X$ is at most $(q+1)/2$.
\end{prop}

\begin{proof}
Consider the set $\mathcal{S}$ of all cocliques of $\Delta_X$ described by Lemma \ref{coclique2}; so $\ell$ is a variable tangent line on $X$,
and $Y$ is a variable external point on $X$, and the coclique consists of all the points of $\mathbb{V}$ on $\cone(Y)$. There are $(q+1)(q-1)=q^2-1$ cocliques in $\mathcal{S}$.
We first show that if $X$ and $X'$ are two external points in the same isometry class, and such that $XX'$ is an external line to $\quadric$, there are precisely $2(q-1)$ points $Y$ in $\cone(X)\cap\cone(X')$.
From Lemma \ref{sameclass}, we know that $|\cone(X)\cap\cone(X')\cap\pi|=2$ for every tangent plane $\pi$ to $\quadric$, not containing $X$ nor $X'$. We double count incident pairs $(P,\pi)$ where $P$ is an external point in $\cone(X)\cap\cone(X')$ and $\pi$ is a tangent plane containing $P$ and not through $X$ nor $X'$. A point of $\cone(X)\cap\cone(X')$ does not lie in a tangent plane through the line $XX'$, and hence, lies on precisely $q-1$ tangent planes, not through $X$ nor $X'$. This shows that there are 
$|\cone(X)\cap\cone(X')|(q-1)$ such pairs.
On the other hand, there are $q^2+1-2(q-1)-2=(q-1)^2$ tangent planes to $\quadric$ not containing $X$ nor $X'$, and each of them contains precisely $2$ points of $\cone(X)\cap \cone(X')$. It follows that $|\cone(X)\cap\cone(X')|(q-1)=2(q-1)^2$ and our claim follows.

Now consider a clique $K$ in $\Delta_X$ and count incident pairs $(P,C)$ where $P\in K$ and $C\in \mathcal{S}$ is a coclique containing $P$. On one hand, this number is at most $|\mathcal{S}|$ since every coclique $C\in \mathcal{S}$ contains at most one point of $K$. The number of cocliques of $\mathcal{S}$ through a point $P$ of $\mathbb{V}$ is the number of choices for a point $Y$ on a tangent line through $X$ such that $YP$ is a tangent line. In other words, the number of points in $\cone(X)\cap \cone(P)$, which we have shown to be $2(q-1)$. 
It follows that $$2(q-1)|K|\leq q^2-1,$$ and hence, $|K|\leq (q+1)/2$.
\end{proof}

Since $\Gamma_X$ is a subgraph of $\Delta_X$, we get the following corollary.

\begin{cor}
    The clique number of $\Gamma_X$ is at most $(q+1)/2$.
\end{cor}

\begin{lemma}\label{coneimpliesexternal}
Let $X$, $A$, $B$ be noncollinear even external points such that $XA$ and $XB$ are external lines.
If $\cone(X)\cap \cone(A)\cap \cone(B)$ has no external points, then $AB$ is an external line.
\end{lemma}

\begin{proof}
Let $X=\langle 1\rangle$, $A=\langle \alpha\rangle$,
$B=\langle \beta\rangle$. From Lemma \ref{extviafield} and Proposition \ref{conecondition}, we have the following facts:
\begin{enumerate}
\item $\tr(\alpha^2)=a^2$ for some $a\in\F_q$.
\item $\tr(\beta^2)=b^2$ for some $b\in\F_q$.
\item $\tr(\alpha)^2-4a^2$ is nonsquare.
\item $\tr(\beta)^2-4b^2$ nonsquare.
\item $-\tr(f(x,y)^2)$ is square for all $(x,y)=(2\alpha\pm a,2\beta\pm b)$.
\end{enumerate}
We need to show $\tr(\alpha\beta)^2-a^2b^2$ is nonsquare.
Now for $x=2\alpha+a$ and $y=2\beta+b$, we have:
\begin{align*}
&\tr(x)=2\tr(\alpha)+4a, &&\tr(y)=2\tr(\beta)+4b,\\
&\tr(x^2)=2a\tr(x), &&\tr(y^2)=2b\tr(y).
\end{align*}
So \begin{align*}
-\tr(f(x,y)^2)&=2\tr(x)\tr(y)\tr(xy)-\tr(x^2)\tr(y)^2-\tr(y^2)\tr(x)^2,\\
&=2\tr(x)\tr(y)(\tr(xy)-a\tr(y)-b\tr(x)),
\end{align*}
Using that $\tr(xy)=4\tr(\alpha\beta)+2b\tr(\alpha)+2a\tr(\beta)+4ab$,
we find
\begin{align*}
-\tr(f(x,y)^2)&=2\tr(x)\tr(y)(4\tr(\alpha\beta)+2b\tr(\alpha)+2a\tr(\beta)+4ab-a\tr(y)-b\tr(x))\\
&=2\tr(x)\tr(y)(4\tr(\alpha\beta)+2b\tr(\alpha)+2a\tr(\beta)+4ab-a(2\tr(\beta)+4b)-b(2\tr(\alpha)+4a))\\
&=2\tr(x)\tr(y)(4\tr(\alpha\beta)-4ab)\\
&=8\tr(x)\tr(y)(\tr(\alpha\beta)-ab)\\
&=32(\tr(\alpha)+2a)(\tr(\beta)+2b)(\tr(\alpha\beta)-ab).
\end{align*}
By this argument, we have
\begin{align}
2(\tr(\alpha)+2a)(\tr(\beta)+2b)(\tr(\alpha\beta)-ab)\in&\square.\label{eq:coneconditionsimple}
\end{align}
Repeating the argument for $x=2\alpha+a$, $y=2\beta-b$ yields that
\begin{align*}
2(\tr(\alpha)+2a)(\tr(\beta)-2b)(\tr(\alpha\beta)+ab)\in&\square.
\end{align*}
Taking the product of these two equations yields
$(\tr(\beta)^2-4b^2)(\tr(\alpha\beta)^2-a^2b^2)\in \square$
and so $\tr(\alpha\beta)^2-a^2b^2$ is nonsquare because 
$\tr(\beta)^2-4b^2$ is nonsquare.
\end{proof}

\begin{cor}\label{conecorollary}
Let $X=\langle 1\rangle$, $A=\langle\alpha\rangle$, $B=\langle \beta\rangle$ with $\tr(\alpha^2)=a^2$ and $\tr(\beta^2)=b^2$.
Suppose that $XA$, $XB$, and $AB$ are external lines to $\quadric$.
Then the cones defined by $X,A,B$ have no common (external) point if and only if
$X,A,B$ are noncollinear and
\[
2(\tr(\alpha)+2a)(\tr(\beta)+2b)(\tr(\alpha\beta)-ab)\in\square.
\]
\end{cor}

\begin{proof}
In the proof of Lemma \ref{coneimpliesexternal}, the 
cones defined by $X,A,B$ have a common (external) point
if and only if $X,A,B$ are noncollinear and the following four expressions hold:
\begin{align*}
2(\tr(\alpha)+2a)(\tr(\beta)+2b)(\tr(\alpha\beta)-ab)\in&\square,\\
2(\tr(\alpha)-2a)(\tr(\beta)+2b)(\tr(\alpha\beta)+ab)\in&\square,\\
2(\tr(\alpha)-2a)(\tr(\beta)-2b)(\tr(\alpha\beta)-ab)\in&\square,\\
2(\tr(\alpha)+2a)(\tr(\beta)-2b)(\tr(\alpha\beta)+ab)\in&\square.
\end{align*}
(See Equation \ref{eq:coneconditionsimple}.)
Now $AB$ is external and so $\tr(\alpha\beta)^2-a^2b^2$ is nonsquare (by Lemma \ref{extviafield}).
Similarly, $\tr(\alpha)^2-4a^2$ and $\tr(\beta)^2-4b^2$ are nonsquare because
$XA$ and $XB$ are external. The product of the first two terms above is
\[
4(\tr(\alpha)^2-4a^2)(\tr(\beta)+2b)^2(\tr(\alpha\beta)^2-a^2b^2)
\]
which is a product of a square, nonsquare, square, and nonsquare. So it is a square
and hence the first expression is square if and only if the second expression is square. A
similar argument holds for any pair of the four expressions. So they are all square
if one of them is square. They are all nonsquare if one of them is nonsquare.
So we can reduce the four expressions to one of them.
\end{proof}

\begin{lemma}\label{lemmacollinear}
 Suppose $A$, $B$, $C$ are collinear points of $\even$, on an external line not through $X$, 
 such that $XA$, $XB$, $XC$
 are external lines. Then at least one of
$\cone(X) \cap \cone(A) \cap \cone(B)$, $\cone(X) \cap \cone(B) \cap \cone(C)$, or
$\cone(X) \cap \cone(C) \cap \cone(A)$ is nonempty.
\end{lemma}

\begin{proof}
Write $X=\langle 1\rangle$, $A=\langle \alpha\rangle$, $B=\langle\beta\rangle$, $C=\langle \alpha+\lambda \beta\rangle$ where $\lambda \in \F_q$. There exist $a,b,c\in\F_q$ such that $\tr(\alpha^2)=a^2$, $\tr(\beta^2)=b^2$, $\tr((\alpha+\lambda \beta)^2)=c^2$.
We will suppose all of $\cone(X) \cap \cone(A) \cap \cone(B)$, $\cone(X) \cap \cone(B) \cap \cone(C)$, or
$\cone(X) \cap \cone(C) \cap \cone(A)$ are empty.

Using Corollary \ref{conecorollary} we know that $\cone(X) \cap \cone(A) \cap \cone(B)$ has no external common points if and only if 
`$2(\tr(\alpha)+2a)(\tr(\beta)+2b)(\tr(\alpha\beta)-ab)\in\square$' where we took $(x,y)=(2\alpha+a,2\beta +b)$.
For $\cone(X) \cap \cone(B) \cap \cone(C)$, and
$\cone(X) \cap \cone(C) \cap \cone(A)$, we will take the equations pertaining to the pairs $(2\alpha+a,2(\alpha+\lambda \beta)+c)$
and $(2\beta+b,2(\alpha+\lambda \beta)-c)$. (Notice the judicious choice to take $-c$ in the second equation, which we
can do, because we can choose any of our four equations to give the cone condition.)
So,
\begin{align}
&2(\tr(\alpha)+2a)(\tr(\alpha)+\lambda \tr(\beta)+2c)(a^2+\lambda \tr(\alpha\beta)-ac)\in \square,\label{eqcone1}\\
&2(\tr(\beta)+2b)(\tr(\alpha)+\lambda \tr(\beta)-2c)(\tr(\alpha\beta)+\lambda b^2+bc)\in \square,\label{eqcone2}\\
&2(\tr(\alpha)+2a)(\tr(\beta)+2b)(\tr(\alpha\beta)-ab)\in \square.\label{eqcone3}
\end{align}
Let $\gamma=\tr(\alpha\beta)-ab$. Since $a^2+2\lambda \tr(\alpha\beta)+\lambda^2b^2=c^2$, it follows that
$2\lambda \gamma=c^2-(a+\lambda b)^2=(c-a-\lambda b)(c+a+\lambda b)$. 
Expression \eqref{eqcone3} is just $2(\tr(\alpha)+2a)(\tr(\beta)+2b)\gamma \in \square$.
Now the product of \eqref{eqcone1} and \eqref{eqcone2} leads to 
\[
4(\tr(\alpha)+2a)(\tr(\beta)+2b)((\tr(\alpha)+\lambda \tr(\beta))^2-4c^2)(a^2+\lambda \tr(\alpha\beta)-ac)(\tr(\alpha\beta)+\lambda b^2+bc)\in \square.
\]
Now 
\begin{align*}
    &(a^2+\lambda \tr(\alpha\beta)-ac)(\tr(\alpha\beta)+\lambda b^2+bc)\\
    =&ab(a+\lambda b -c +\gamma \lambda a^{-1})(a+\lambda b +c +\gamma b^{-1})\\
    =&ab( (a+\lambda b)^2-c^2 +\gamma b^{-1}(a+\lambda b-c)+\gamma \lambda a^{-1}(a+\lambda b+c)+\lambda\gamma^2a^{-1}b^{-1})\\
    =& ab( -2\lambda \gamma +\gamma b^{-1}(a+\lambda b-c)+\gamma \lambda a^{-1}(a+\lambda b+c)+\lambda\gamma^2a^{-1}b^{-1})\\
    =& \gamma \left( -2\lambda ab + a(a+\lambda b-c)+ \lambda b(a+\lambda b+c)+\lambda\gamma \right)\\
    =& \gamma \left( a(a-c)+ \lambda b(\lambda b+c)+\lambda\gamma \right)\\ 
    =& \gamma \left( a(a-c)+ \lambda b(\lambda b+c)+\tfrac{1}{2}(c^2-(a+\lambda b)^2)\right)\\ 
    =& \frac{1}{2}\gamma (c-a + \lambda b )^2.
\end{align*}
Therefore,
\[
2(\tr(\alpha)+2a)(\tr(\beta)+2b) \gamma \cdot ((\tr(\alpha)+\lambda \tr(\beta))^2-4c^2) \in \square
\]
and hence $(\tr(\alpha)+\lambda \tr(\beta))^2-4c^2 \in \square$,
which is a contradiction as $XC$ is an external line.
\end{proof}

\begin{lemma}\label{lemmanoncollinear}
Let $A,B,C$ be three noncollinear points of $\even$ such that $XA$, $XB$, $XC$ are external lines. 
If the plane spanned by $A,B,C$ is a tangent plane, then at least one of
$\cone(X) \cap \cone(A) \cap \cone(B)$, $\cone(X) \cap \cone(B) \cap \cone(C)$, or
$\cone(X) \cap \cone(C) \cap \cone(A)$ is nonempty.
\end{lemma}

\begin{proof}
Write $X=\langle 1\rangle$, $A=\langle \alpha\rangle$, $B=\langle\beta\rangle$, $C=\langle \gamma\rangle$. 
By Lemma \ref{extviafield}, there exist $a,b,c\in\F_q$ such that $\tr(\alpha^2)=a^2$, $\tr(\beta^2)=b^2$, $\tr(\gamma^2)=c^2$.
We will suppose all of $\cone(X) \cap \cone(A) \cap \cone(B)$, $\cone(X) \cap \cone(B) \cap \cone(C)$, or
$\cone(X) \cap \cone(C) \cap \cone(A)$ are empty. So (by Proposition \ref{conecondition} and Corollary \ref{conecorollary}) we have:
\begin{align*}
&2 (\tr(\alpha) + 2 a) (\tr(\beta) + 2 b) (\tr(\alpha\beta) - a b) \in \square,\\
&2 (\tr(\beta) - 2 b) (\tr(\gamma) - 2 c) (\tr(\beta\gamma) - b c) \in \square,\\
&2 (\tr(\gamma) + 2 c) (\tr(\alpha) - 2 a) (\tr(\alpha\gamma) + a c) \in \square.  
\end{align*}
The product of these three expressions, divided by 4, is also a square:
\begin{equation}
2(\tr(\alpha)^2 -4a^2) (\tr(\beta)^2 - 4 b^2)(\tr(\gamma)^2 - 4 c^2) 
(\tr(\alpha\beta) - a b) (\tr(\beta\gamma) - b c)(\tr(\alpha\gamma) + a c) \in \square.
\label{tripleproduct}
\end{equation}
Since $ABC$ is a tangent plane, we also have (by Lemma \ref{tangentplane}),
\begin{equation}
2\tr(\alpha\beta)\tr(\beta\gamma)\tr(\gamma\alpha)=
-a^2b^2c^2+a^2\tr(\beta\gamma)^2+b^2\tr(\alpha\gamma)^2+
c^2\tr(\alpha\beta)^2. \label{tangentcondition}
\end{equation}
Substituting $2\tr(\alpha\beta)\tr(\beta\gamma)\tr(\gamma\alpha)$ as expressed in \eqref{tangentcondition} into
\eqref{tripleproduct} yields
\[
(\tr(\alpha)^2 -4a^2) (\tr(\beta)^2 - 4 b^2)(\tr(\gamma)^2 - 4 c^2) 
(a b c-a \tr(\beta\gamma)+b \tr(\gamma\alpha)-c \tr(\alpha\beta))^2\in\square
\]
and hence 
\[
(\tr(\alpha)^2 -4a^2) (\tr(\beta)^2 - 4 b^2)(\tr(\gamma)^2 - 4 c^2) 
\]
is a square -- a contradiction, since all three terms are nonsquare by Lemma \ref{extviafield}.
\end{proof}

We now prove Theorem \ref{maintheorem} that Bruen chains can be equivalently
formulated as $(q+1)/2$-cliques of $\Gamma_X$ and $\Delta_X$. The following theorem also gives more information about the cliques of smaller size contained in $\Gamma_X$ and $\Delta_X$.

\begin{theorem}\label{maintheoremextra}
All cliques of $\Gamma_X$ define caps. Furthermore, cliques of size $(q+1)/2$ of $\Gamma_X$ are in one-to-one correspondence with Bruen chains.
All cliques of $\Delta_X$ define caps or collinear point sets on a line through $X$. Furthermore cliques of size $(q+1)/2$ of $\Delta_X$ are in one-to-one correspondence with Bruen chains.    
\end{theorem}

\begin{proof} 
We will use the symbol $\sim$ and $\not\sim$ for `adjacent' and `nonadjacent' (respectively), for two vertices of a graph, that will be apparent from the context.
We will first show first that Bruen chains give rise to cliques of $\Delta_X$ and $\Gamma_X$. So let $\mathcal{B}$ be a Bruen chain (seen as points external to $\quadric$). By Corollary \ref{Heden_main}, all points of $\mathcal{B}$ are of the same isometry type and we may suppose without loss of generality that all points are in $\even$, and that $X=\langle 1\rangle$ is in $\mathcal{B}$. The set $\mathcal{B}\setminus\{ X\}$ has size $\frac{q+1}{2}$ and by Lemma \ref{bruenPerp}, consists of points $Y$ such that $XY$ is an external line.  Hence, all points of $\mathcal{B}$ are in $\mathbb{V}$, that is, they are vertices of $\Gamma_X$ and $\Delta_X$. From Lemma \ref{bruenPerp}(ii) it follows that if $A,B\in \mathcal{B}$, then $AB$ is not incident with $X$. Furthermore,  by Corollary \ref{Bruenimpliescone}, $\cone(A)\cap\cone(B)\cap\cone(X)=\varnothing$ for all $A,B\in \mathcal{K}\setminus\{X\}$, so the points of $\mathcal{B}\setminus \{X\}$ form a clique in $\Gamma_X$ and in $\Delta_X$.

Now let $\mathcal{K}$ be a clique of $\Delta_X$ and let $A,B,C$ be collinear points of $\mathcal{K}$ on a line $\ell$ of $\PG(3,q)$ where $X$ does not lie on $\ell$. Since $A,B$ are vertices of $\Delta_X$, we have that $XA,XB$ are external lines. Furthermore, since and $A\sim B$, we then have that $\cone(X)\cap\cone(A)\cap\cone(B)=\varnothing$ so $\ell=AB$ is an external line by Lemma \ref{coneimpliesexternal}. However, by Lemma \ref{lemmacollinear}, we have that $A\not\sim B$, or $A\not\sim C$ or $B\not\sim C$, a contradiction. We conclude that it is impossible that three collinear points $A,B,C$ on a line not through $X$ are contained in a clique $\mathcal{K}$ of $\Delta_X$.
Since $\Gamma_X$ is a subgraph of $\Delta_X$, the same conclusion holds for $\Gamma_X$: three collinear points $A,B,C$ on a line not through $X$ cannot be contained in a clique of $\Gamma_X$. Furthermore, since $\Gamma_X$ does not contain vertices $A,B$ collinear with $X$, we find that a clique of $\Gamma_X$ is necessarily a cap.
Now suppose that a clique of $\Delta_X$ contains a collinear subset $A,B,C$ on a line $\ell$ through $X$, and some point $D\notin \ell$. Since $A$ and $D$ are vertices of $\Delta_X$ we have that $\ell=XA$ and $XD$ are external lines. Furthermore, since $A\sim D,B\sim D,C\sim D$, Lemma \ref{coneimpliesexternal} shows that $AD,BD,CD$ are external lines. However, then Lemma \ref{lemmacollinear} again gives a contradiction.
This proves that every clique of $\Delta_X$ is either a set of points, all collinear on a line through $X$, or a cap.
Finally, let $\mathcal{K}'$ be a clique of $\Delta_X$ of size $\frac{q+1}{2}$. Then it is impossible for all points of $\mathcal{K}'$ to be collinear on a line $\ell$ through $X$ since every point of $\mathcal{K}'$ is an even point on $\ell$, distinct from $X$, and there are only $\frac{q-1}{2}$ such points on $\ell$ (the external line $\ell$ contains precisely $\frac{q+1}{2}$ points of $\mathbb{E}$, one of which is $X$).
This shows that the clique $\mathcal{K}'$ is a cap and it follows that $\mathcal{K}'$ is a clique of $\Gamma_X$ as well.
We will show that $\mathcal{K}'$ is a Bruen chain by showing that the points of $\mathcal{K}'$ satisfy the conditions of Lemma \ref{bruenPerp}.
First, by Lemma \ref{coneimpliesexternal}, $AB$ is an external line for each 
$A,B\in\mathcal{K}$, so it remains to show that $ABC$ is a secant plane, for each $A,B,C\in\mathcal{K}$. We have shown already that the points of $\mathcal{K}'$ form a cap, so $A,B,C$ indeed span a plane, which is not a tangent plane by Lemma \ref{lemmanoncollinear}.
Therefore, $\mathcal{K}'\cup \{X\}$ is a Bruen chain.
\end{proof}

\begin{cor}
A Bruen chain has an algebraic characterisation within the finite field $\F_{q^4}$.
In particular, it is equivalent to a set of $(q+1)/2$ elements $S$ of 
$\F_{q^4}\backslash \F_q$ such that for all $\alpha,\beta\in S$:
\begin{enumerate}[(i)]
\item $(\alpha^q-\alpha)/(\beta^q-\beta) \notin \F_q$;
\item $\tr(\alpha)^2-4a^2$ is a nonsquare;
\item $2(\tr(\alpha)+2a)(\tr(\beta)+2b)(\tr(\alpha\beta)-ab)$ is a nonzero square;
\end{enumerate}
where $a$ and $b$ are elements of $\F_q$ such that $\tr(\alpha^2)=a^2$ and $\tr(\beta^2)=b^2$.
\end{cor}

\begin{proof} 
Throughout this proof, let $A=\langle \alpha\rangle$ and $B=\langle\beta\rangle$ where $\alpha,\beta\in S$.

Theorem \ref{maintheoremextra} shows that a Bruen chain is equivalent to a clique of $\Gamma_X$ of size $\frac{q+1}{2}$. 
Now let $S$ be a set of $\frac{q+1}{2}$ field elements of $\F_{q^4}\setminus \F_q$ and suppose that $S$ defines a clique of $\Gamma_X$, then every element $\alpha$ of $S$ satisfies  (ii): $\tr(\alpha^2)=a^2$ where $\tr(\alpha^2)-4a^2$ is a nonsquare since $A$ must be an even point such that $XA$ is an external line. Furthermore, if $\alpha,\beta\in S$, then $A,B$ and $X=\langle 1\rangle$ need to be noncollinear, which is equivalent to the condition that $1,\alpha,\beta$ are $\F_q$-independent, which is expressed in (i). 
If $\alpha,\beta\in S$ then $\cone(A)\cap\cone(B)\cap\cone(X)$ does not have points outside $\quadric$. Since $XA$ and $XB$ are external lines for all $A,B$ where $A=\langle \alpha\rangle$, $B=\langle \beta\rangle$ with $\alpha,\beta\in S$, Lemma \ref{coneimpliesexternal} shows that $AB$ is an external line, and then by Lemma \ref{conecorollary}, $\cone(A)\cap\cone(B)\cap\cone(X)$ having no external points shows that (iii) holds.

Vice versa, if $S$ is a set of $\frac{q+1}{2}$  elements satisfying $(i)-(ii)-(iii)$ then it is clear that for all points $A=\langle \alpha\rangle$, $B=\langle \beta\rangle$ with $\alpha,\beta$ in $S$ we have that $A,B$ are even points spanning an external line with $X$ by $(ii)$. By (i), $1,A,B$ are not collinear, which shows by Corollary \ref{conecorollary} that (iii) implies  that $\cone(A)\cap\cone(B)\cap\cone(X)$ has no points outside $\quadric$.
\end{proof}

For $41\leqslant q\leqslant \largestnumber$, every $(q-1)/2$-clique of $\Delta_X$ is contained in an external line through $X$, and so Bruen chains do not exist for these values of $q$; and this is our statement for Theorem \ref{computationalresult}.
Using the finite field model above, we could efficiently construct the graph
$\Gamma_X$ in $\textsc{GAP}$ using the package \textsc{Grape} \cite{Grape}.
(We provide this code in Appendix B.)
Without the finite field model, this is memory intensive
and time consuming. We then made sure the graph was equipped with the automorphisms
arising by taking the stabiliser of $X$ in the isometry group of $\quadric$. We were able to compute the clique numbers using the methods described in the following section.

\section{Computational results}

First we list (in Table \ref{tbl:known}) the known examples and their stabilisers. There are 2 examples each for $q\in\{9,25\}$ if we take the isometry group, but they are equivalent under semi-similarities.

\begin{table}[!ht]\footnotesize
\begin{tabular}{lclc}
\toprule
$q$ & Number of Bruen chains & Stabilisers & Example reference\\
\midrule
5 &1& $C_2\times S_4$&\cite{Bruen1978} \\
7 &2& $C_2\times S_4$, $S_5$&\cite{Bruen1978,Korchmaros81}\\
9 &1 & $C_2\times S_4$ &\cite{Abatangelo, HedenSaggese}\\
11 & 2 & $D_{20}$, $C_2\times A_5$&\cite{Korchmaros81}\\
13 & 2 & $S_4$, $C_2\times C_2\times S_3$ &\cite{Heden95}\\
17 & 2 & $D_{12}$, $D_{16}$ &\cite{Heden95}\\
19 & 2 & $D_{12}$, $S_4$ &\cite{Heden95}\\
23 & 1 & $D_8$ &\cite{Heden95}\\
25 & 1 & $D_{12}$ &\cite{HedenSaggese}\\
27 & 1 & $C_3\times A_4$ &\cite{HedenSaggese}\\
31 & 1 & $S_4$ &\cite{Heden95}\\
37 & 1 & $C_5 : C_4$ &\cite{Heden98}\\
\bottomrule \\
\end{tabular}
\caption{Known examples and their stabilisers.}\label{tbl:known}
\end{table}

\vspace{-2ex}
Next, we list clique numbers of $\Gamma_X$ in Table \ref{tbl:cliquenumbers}, which also shows that no Bruen chain
exists for $41\le q\le \largestnumber$.
We plot the clique numbers of $\Gamma_X$ against $q$ for $5 \leqslant q \leqslant 97$ in Figure \ref{plot_clique_numbers}. The initial linear trend quite clearly changes once $q$ is sufficiently large, adding evidence to the conjectures of \cite{Heden98,HedenSaggese} that there are no further Bruen chains for $q>37$.

To determine the clique numbers for each $\Gamma_X$, the graph was constructed using GAP and its
package \textsc{Grape} \cite{Grape}. The \emph{Orderly Algorithm} \cite{orderly} implemented with the the 
\emph{MinimalImage}
function of the GAP package \textsc{images} \cite{images1.3.1} was employed to produce starter sets of size $1$ or $2$ for possible cliques.
We exported each induced subgraph of $\Gamma_X$ 
 to a DIMACS format text file,
and we then ran a C++ program that uses the \emph{MaxCliqueDyn} algorithm, known
as \texttt{mcqd} \cite{mcqd}. For $89 \leq q \leq 97$ a parallel implementation \cite{mcqdpara} was used on a server with 64 cores. For $q=89$, it took approximately two weeks using \cite{mcqd} and approximately $8$ days using \cite{mcqdpara}. It took approximately $3$ weeks to complete $q=97$ using \cite{mcqdpara}.

For small values of $q$, we were also able to determine the clique number directly using the clique algorithm provided by \textsc{GRAPE}. In the case of existence, the \textsc{GRAPE} functionality coupled with starter sets was sufficient for enumeration of the Bruen chains.

In Lemma \ref{coclique2}, we showed that the coclique number of $\Delta_X$ is $q(q-1)/2$, and so the clique number of $\Gamma_X$ must be at least as large, since it is a spanning subgraph. For $q=29$ we computed the coclique number and determined it to be exactly $q(q-1)/2 = 406$. We had hoped that it might exceed this value and, if sufficiently large, that it might be shown to prohibit the existence of a Bruen chain. Since this approach seemed unpromising and given it is computationally difficult to determine the coclique number for $\Gamma_X$, we did not attempt other values of $q$.

\begin{table}[!ht]\footnotesize
\begin{tabular}{ll|ll|ll|ll|ll|ll}
\toprule
$q$ & $\omega(\Gamma_X)$ &
$q$ & $\omega(\Gamma_X)$ &
$q$ & $\omega(\Gamma_X)$ &
$q$ & $\omega(\Gamma_X)$ &
$q$ & $\omega(\Gamma_X)$ &
$q$ & $\omega(\Gamma_X)$\\
\midrule
 29 & 9 & 41 & 13 & 49 & 13 & 61 & 15 & 73 & 15 &  83 & 15\\
 31 & 16 &  43 & 12 & 53 & 13 & 67 & 14 & 79 & 16 &  89 & 15 \\
 37 & 19 & 47 & 12 &  59 & 15 & 71 & 15 & 81 & 15  & 97 & 17\\
 \bottomrule \\
\end{tabular}
\caption{Clique numbers of $\Gamma_X$, for $29\le q\le \largestnumber$.}\label{tbl:cliquenumbers}
\end{table}

\begin{figure}[!ht]
\centering
\begin{tikzpicture}[scale=0.75]
\begin{axis}[width=.8\textwidth,xlabel=$q$,ylabel=$\omega(\Gamma_X)$, 
    xmin = 3, xmax=100, xtick distance=5, ymin=2, ymax=20, ytick distance = 2,
	yticklabel style={
            /pgf/number format/fixed,
            /pgf/number format/precision=0,
            /pgf/number format/fixed zerofill
        }]
  \addplot [color=blue, 
           only marks, 
           mark=*, 
           mark size=1.5pt]
coordinates {
(5, 3)
(7, 4)
(9, 5)
(11, 6)
(13, 7)
(17, 9)
(19, 10)
(23, 12)
(25, 13)
(27, 14)
(29, 9)
(31, 16)
(37, 19)
(41, 13)
(43, 12)
(47, 12)
(49, 13)
(53, 13)
(59, 15)
(61, 15)
(67, 14)
(71, 15)
(73, 15)
(79, 16)
(81, 15)
(83, 15)
(89, 15)
(97, 17)
 };
\end{axis}
\end{tikzpicture}\vspace{-2ex}
\caption{Clique numbers of $\Gamma_X$ for $5\le q\le \largestnumber$.} \label{plot_clique_numbers}
\end{figure}

\section*{Acknowledgements}
We would like to thank Gordon Royle for his valuable input. The first and second authors were supported by the Australian Research Council Discovery Grant DP200101951 and the third author was supported by the Marsden Fund Council administered by the Royal Society of New Zealand.

\section*{Appendix A: Representatives of known Bruen chains}

For each of the values of $q\le 37$, we give the complete set of Bruen chains up to equivalence.
We will use the finite field model presented in Section \ref{sect:model}.
Each list of integers (Table \ref{tbl:allBruen}) is a set of exponents of the primitive element of $\mathbb{F}_q$,
which in GAP notation, is \texttt{Z(q\^{}4)}. A $\star$ indicates that under the full collineation group,
we can disregard the extra example because it is equivalent under semi-similarities to another. Each example begins with `0'; so for larger $q$, some lines of the table need to carry-over, so one can delineate the examples
by knowing this.

\vspace{-1ex}
\begin{table}[!ht]
{\scriptsize
\begin{tabular}{ll}
\toprule 
$q$ & Bruen chains (up to isometry) \\
\midrule
5 & 0, 182, 305, 605\\
7 & 0, 234, 1397, 1640, 1818\\
&  0, 774, 1126, 1818, 2282 \\
9 & 0, 44, 114, 3669, 6035, 6222\\
 & 0, 77, 2613, 4396, 4447, 6035 $\star$\\
11 & 0, 666, 1091, 6507, 10658, 11079, 13791\\
& 0, 1621, 4635, 5091, 8026, 8790, 9359 \\
13 &  0, 1201, 3974, 9088, 9209, 20275, 25144, 26421 \\
&  0, 9088, 9209, 9712, 15691, 20275, 25061, 27005 \\ 
17 & 0, 11064, 23365, 44028, 51732, 54288, 55375, 74434, 75024, 81544\\ 
& 0, 8573, 17348, 40937, 42498, 44028, 65239, 69987, 72812, 76936\\ 
19 & 0, 2358, 17832, 30167, 50702, 53413, 57124, 63144, 81317, 103363, 126708\\
& 0, 4998, 34061, 39278, 51323, 51469, 74971, 103689, 104039, 113312, 126708\\
23 & 0, 40351, 50476, 65922, 141369, 144680, 151517, 179361, 193641, 196170, 
  226383, 250126, 258931 \\ 
25 & 0, 7765, 21365, 51750, 52028, 60937, 179761, 187904, 222287, 252075, 
  263033, 272801, 277793, 360244\\ 
& 0, 44967, 119676, 125676, 169013, 178576, 187904, 276812, 292945, 305987, 
  356702, 370842, 380493, 384441 $\star$\\
27 & 0, 15789, 29254, 40015, 41748, 43717, 63057, 70717, 173045, 255719, 278951, 
  281658, 335753, 430538, 493601\\
31 & 0, 6510, 12459, 35613, 282883, 288878, 298214, 398628, 405647, 519521, 
  570613, 588163, 631263, 769878,\\ 
  &\quad  848886, 900936, 906588 \\
37 & 0, 91942, 225190, 305966, 396584, 452161, 490392, 495825, 594231, 735337, 
  805164, 849785, 858281, 1137850,\\ 
  & \quad   1188868, 1259696, 1287313, 1325949, 1476150, 1846082\\
\bottomrule \\
\end{tabular}
}
\caption{All Bruen chains represented as elements of $\mathbb{F}_{q^4}$, for $q\le 37$.}\label{tbl:allBruen}
\end{table}

\newpage
\section*{Appendix B: GAP code}

\begin{lstlisting}
LoadPackage("fining");

# Construct \Gamma_X for a given value of q
Gamma_X := function(q)

  local capcondition, conecondition, vec, basis, tr, gram, form, eq, squares, nonsquares, x, ext, ext_field, g, permgroup, permgroup1, good, gamma, permgroup2, perm, gens, aut;

  conecondition := function(alpha, beta, q)
    local tr, tra, trb, a, b, f, sq;
    if ForAny(GF(q), t -> beta - t * alpha in GF(q)) then
      return true;
    fi;
    tr := x -> x+x^q+x^(q^2)+x^(q^3);;
    tra := tr(alpha^2);
    trb := tr(beta^2);
    a := First(GF(q), t -> t^2 = tra);
    b := First(GF(q), t -> t^2 = trb);
    f := {x,y} -> tr(y) * x - tr(x) * y;
    sq := AsList(Group(Z(q)^2));;
    return ForAll(Cartesian([2*alpha+a,2*alpha-a],[2*beta+b,2*beta-b]), t -> -tr(f(t[1],t[2])^2) in sq );
  end;

  capcondition := {b,c} -> not (ext_field[b]^q-ext_field[b])/(ext_field[c]^q-ext_field[c]) in GF(q);

  vec := AsVectorSpace(GF(q), GF(q^4));;
  basis := Basis(vec);;
  tr := x -> x+x^q+x^(q^2)+x^(q^3);;
  gram := List(basis, t -> List(basis, u -> tr(t*u)));;
  ConvertToMatrixRep(gram,GF(q));
  form := BilinearFormByMatrix(gram,GF(q));;
  eq := PolarSpace(form);;
  squares := Union(AsList(Group(Z(q)^2)), [0*Z(q)]);;
  nonsquares := Difference(AsList(GF(q)),squares);;
  x := VectorSpaceToElement(PG(3,q),[1,0,0,0]*Z(q)^0);
  ext := AsList(FiningOrbit(IsometryGroup(eq),x));;
  ext_field := List(ext, t -> basis * Unpack(t!.obj));;
  g := IsometryGroup(eq);;
  permgroup := Action(g, ext);
  permgroup1 := Stabiliser(permgroup,1);
  good := Filtered([2..Size(ext)], t -> tr(ext_field[t])^2-4*tr(ext_field[t]^2) in nonsquares );;

  gamma:=Graph(permgroup1,good,OnPoints,{b,c} -> b<>c and capcondition(b,c) and conecondition(ext_field[b],ext_field[c],q));;

  permgroup2 := Action(permgroup1, good);
  perm:=PermListList(VertexNames(gamma), good);
  gens:=GeneratorsOfGroup(permgroup2);
  gens:=List(gens, t -> t^perm);
  aut:=Group(gens);
  gamma!.autGroup := aut;

  return gamma;
end;

# Construct \Delta_X for a given value of q
Delta_X := function(q)

  local conecondition, vec, basis, tr, gram, form, eq, squares, nonsquares, x, ext, ext_field, g, permgroup, permgroup1, good, gamma, permgroup2, perm, gens, aut;

  conecondition := function(alpha, beta, q)
    local tr, tra, trb, a, b, f, sq;
    if ForAny(GF(q), t -> beta - t * alpha in GF(q)) then
      return true;
    fi;
    tr := x -> x+x^q+x^(q^2)+x^(q^3);;
    tra := tr(alpha^2);
    trb := tr(beta^2);
    a := First(GF(q), t -> t^2 = tra);
    b := First(GF(q), t -> t^2 = trb);
    f := {x,y} -> tr(y) * x - tr(x) * y;
    sq := AsList(Group(Z(q)^2));;
    return ForAll(Cartesian([2*alpha+a,2*alpha-a],[2*beta+b,2*beta-b]), t -> -tr(f(t[1],t[2])^2) in sq );
  end;

  vec := AsVectorSpace(GF(q), GF(q^4));;
  basis := Basis(vec);;
  tr := x -> x+x^q+x^(q^2)+x^(q^3);;
  gram := List(basis, t -> List(basis, u -> tr(t*u)));;
  ConvertToMatrixRep(gram,GF(q));
  form := BilinearFormByMatrix(gram,GF(q));;
  eq := PolarSpace(form);;
  squares := Union(AsList(Group(Z(q)^2)), [0*Z(q)]);;
  nonsquares := Difference(AsList(GF(q)),squares);;
  x := VectorSpaceToElement(PG(3,q),[1,0,0,0]*Z(q)^0);
  ext := AsList(FiningOrbit(IsometryGroup(eq),x));;
  ext_field := List(ext, t -> basis * Unpack(t!.obj));;
  g := IsometryGroup(eq);;
  permgroup := Action(g, ext);
  permgroup1 := Stabiliser(permgroup,1);
  good := Filtered([2..Size(ext)], t -> tr(ext_field[t])^2-4*tr(ext_field[t]^2) in nonsquares );;

  gamma:=Graph(permgroup1,good,OnPoints,{b,c} -> b<>c and conecondition(ext_field[b],ext_field[c],q));;

  permgroup2 := Action(permgroup1, good);
  perm:=PermListList(VertexNames(gamma), good);
  gens:=GeneratorsOfGroup(permgroup2);
  gens:=List(gens, t -> t^perm);
  aut:=Group(gens);
  gamma!.autGroup := aut;

  return gamma;
end;
\end{lstlisting}

\bibliographystyle{abbrv}
\bibliography{references}

\end{document}